\documentclass[12pt]{amsart}
\usepackage{amsmath}
\usepackage{amssymb}
\usepackage{comment}
\newtheorem{thm}{Theorem}[section]

\newtheorem{lem}[thm]{Lemma}
\newtheorem{example}[thm]{Example}

\theoremstyle{definition}

\theoremstyle{remark}
\newtheorem{rem}[thm]{\bf{Remark}}
\numberwithin{equation}{section}

\newcommand{\beas} {\begin{eqnarray*}}
\newcommand{\eeas} {\end{eqnarray*}}
\newcommand{\bes} {\begin{equation*}}
\newcommand{\ees} {\end{equation*}}
\newcommand{\be} {\begin{equation}}
\newcommand{\ee} {\end{equation}}
\newcommand{\bea} {\begin{eqnarray}}
\newcommand{\eea} {\end{eqnarray}}
\newcommand{\ra} {\rightarrow}
\newcommand{\txt} {\textmd}

\newcommand{\R}{\mathbb R}

\newcommand{\C}{\mathbb C}
\newcommand{\T}{\mathbb T}
\newcommand{\N}{\mathbb N}

\newcommand{\la}{\lambda}
\newcommand{\g}{\mathfrak{g}}

\begin{document}

\title[Theorems of Ingham and Chernoff] {Analogs of certain quasi-analiticity results on Riemannian symmetric spaces of noncompact type}

\author{Mithun Bhowmik, Sanjoy Pusti and Swagato K. Ray}

\address{(Mithun Bhowmik, Sanjoy Pusti) Department of Mathematics, IIT Bombay, Powai, Mumbai-400076, India}
\email{mithun@math.iitb.ac.in, sanjoy@math.iitb.ac.in}

\address{(Swagato K Ray) Stat-Math Unit, Indian Statistical Institute, 203, B.T. Road, Kolkata-700108, India}
\email{swagato@isical.ac.in}

\thanks{The first author is supported by the Post Doctoral fellowship from IIT Bombay. Second author is supported partially by SERB, MATRICS, MTR/2017/000235.}
\thanks{ We are grateful to Professors Sayan Bagchi and Angela Pasquale for their helpful comments and suggestions.}

%%% ----------------------------------------------------------------------

\begin{abstract}
An $L^2$ version of the celebrated Denjoy-Carleman theorem regarding quasi-analytic functions was proved by Chernoff \cite{CR} on $\R^d$ using iterates of the Laplacian. In $1934$ Ingham \cite{I} used the classical Denjoy-Carleman theorem to relate the decay of Fourier transform and quasi-analyticity of integrable functions on $\R$. In this paper we extend both these theorems to Riemannian symmetric spaces of noncompact type and show that the theorem of Ingham follows from that of Chernoff.
\end{abstract}

\subjclass[2010]{Primary 43A85; Secondary 22E30, 33C67}

\keywords{Riemannian symmetric space, Quasi-analyticity,  Ingham's theorem}

%%% ----------------------------------------------------------------------
\maketitle
%%% ----------------------------------------------------------------------

\section{Introduction}
In this paper we will concern ourselves with the classical problem of determining the relationship between the rate of decay of the Fourier transform of an integrable function at infinity and the size of support of the function in the context of a Riemannian symmetric space $X=G/K$ of noncompact type, where $G$ is a connected noncompact semisimple Lie group with finite center and $K$ is a maximal compact subgroup of $G$. To understand the context better let us consider the following well known statement: Suppose $f\in L^1(X)$ and its Fourier transform satisfies the estimate
\begin{equation*}
|\tilde{f}(\lambda,k)|\leq C\widehat{h_t}(\lambda),\:\:\:\:\lambda\in\frak a^*,
\end{equation*}
for some $C$ positive. If $f$ vanishes on a nonempty open subset of $X$ then $f$ vanishes identically (we refer the reader to Section 2 for meaning of relevant symbols). Unlike its obvious Euclidean counterpart, the proof of this statement is not quite elementary as it uses a nontrivial result such as the Kotake-Narasimhan theorem (\cite{KN}). However, sharp conditions on the decay of the  Fourier transform
which prohibits a nonzero integrable function to vanish on an open set is known from some classical results proved in \cite{Branges, Hi, I, Koo, L2, PW}. 
Recently, some of these results have appeared again in the context of uniqueness of solution of Schrodinger equation \cite{kenig}. This motivated us to have a fresh look at these results and explore the possibility of extending them beyond Euclidean spaces. Very recently we could extend the main result of \cite{L2} for Riemannian symmetric spaces $X$ (see \cite{BR}) and in this paper our objective is to do the same with the result of Ingham \cite{I}. 
A recently proved analogue of the result of Ingham on $\R^d$ states the following.    
\begin{thm}[\cite{BRS}, Theorem 2.2] \label{ingrn}
Let $\theta:\R^d \rightarrow [0,\infty)$ be a radially decreasing function with $\lim_{\|\xi\|\to\infty}\theta(\xi)=0$ and we set
\be \label{idefn}
I=\int_{\|\xi\|\geq 1}\frac{\theta (\xi)}{\|\xi\|^d}~d\xi.
\ee
\begin{enumerate}
\item[$(a)$] Let $f\in L^p(\R^d)$, $p \in [1,2]$, be such that its Fourier transform $\mathcal F f$ satisfies the estimate
\be \label{ingcond}
|\mathcal Ff(\xi)|\leq Ce^{-\theta (\xi)\|\xi\|}, \:\:\:\: \txt{ for almost every } \xi \in \R^d.
\ee
If $f$ vanishes on a nonempty open set in $\R^d$ and $I$ is infinite, then $f$ is the zero function on $\R^d$. 
\item[$(b)$] If $I$ is finite then given any positive number $L$, there exists a nontrivial radial function $f\in C_c^{\infty}(\R^d)$ supported in $B(0, L)$, satisfying the estimate (\ref{ingcond}).
\end{enumerate}
\end{thm}
Here,
\bes
\mathcal Ff(\xi)=\int_{\R^d}f(x)e^{-2\pi i x\cdot \xi}dx,
\ees
with $x\cdot \xi$ being the Euclidean inner product of the vectors $x$ and $\xi$.
It might be worth noting that the function $\xi\to \theta(\xi)\|\xi\|$ may not be radially increasing and hence the condition (\ref{ingcond}) does not immediately imply a pointwise decay of the Fourier transform. However, divergence of the integral $I$ does  imply that
\begin{equation*}
\limsup_{\|\xi\|\to\infty}\theta(\xi)\|\xi\|=\infty.
\end{equation*} 
For $d=1$, using certain extra assumption on the function $\theta$, Ingham showed that the condition (\ref{ingcond}) together with the divergence of the integral $I$ implies that
\begin{equation*}
\sum_{m=1}^{\infty}\left\|\frac{d^mf}{dx^m}\right\|_{\infty}^{-\frac{1}{m}}=\infty,
\end{equation*}
(see \cite[p. 30]{I} for details). It then follows from the Denjoy-Carleman theorem (\cite[Theorem 19.11]{Ru}) that $f$ is quasi-analytic and hence vanishes identically under the assumption that it vanishes on a nonempty open set. 
It is fairly natural to anticipate that an extension of Ingham's result on $\R^d$ or on a Riemannian symmetric space $X$ will involve a suitable extension of the Denjoy-Carleman theorem for these spaces. One such result was obtained by Bochner in \cite[Theorem 3]{B} which provides an analogue of the Denjoy-Carleman theorem for $\R^d$ involving iterates of the Laplace-Beltrami operator. An important variant of the result of Bochner was later proved by Chernoff in \cite[Theorem 6.1]{CR}. 
\begin{thm}\label{ch} 
Let $f:\R^d\to\C$ be a smooth function.
\begin{enumerate}
\item[(a)] (Bochner) If
\bes
\sum_{m\in \N}\|\Delta_{\R^d}^m f\|_{\infty}^{-\frac{1}{m}}=\infty,
\ees
and for all $m\in\N\cup\{0\}$, $\Delta_{\R^d}^mf(x)=0$ for all $x$ in a set $U$ of analytic determination then $f$ vanishes identically.
\item[(b)] (Chernoff)
Suppose that for all $m\in \N\cup \{0\}$, $\Delta_{\R^d}^m f\in L^2(\R^d)$, and  
\bes
\sum_{m\in \N}\|\Delta_{\R^d}^m f\|_2^{-\frac{1}{2m}}=\infty.
\ees
If $f$ and all its partial derivatives vanish at the origin then $f$ is identically zero.
\end{enumerate}
\end{thm}
In the above $\Delta_{\R^d}$ denotes the usual Laplacian on $\R^d$. It was noted in \cite{B} itself that part $(a)$ of the theorem above generalizes to Riemannian manifolds and hence is applicable on Riemannian symmetric spaces. However, we could not use this result to prove Theorem \ref{ingrn} even for $d=1$ due to the difficulty in estimating the relevant $L^{\infty}$ norms.
On the other hand, part $(b)$ of the theorem above can be suitably used to prove Theorem \ref{ingrn}. 
It is thus natural to look for an analogue of Chernoff's result on Riemannian symmetric spaces of noncompact type and then try to use it to prove an analogue of Theorem \ref{ingrn}. Our first result in this paper is the following weaker version of Theorem \ref{ch}, $(b)$ for a Riemannian symmetric space $X$ of noncompact type with $\text{rank}\:d\geq 1$.  
\begin{thm}\label{cher-symm}
Let $f\in C^\infty(G/K)$ be such that $\Delta^m f\in L^2(G/K)$, for all $m\in \N\cup \{0\}$ and 
\be \label{chernoffcond}
\sum_{m=1}^{\infty} \|\Delta^m f\|_2^{-\frac{1}{2m}}= \infty.
\ee
If $f$ vanishes on any nonempty open set in $G/K$ then $f$ is identically zero.
\end{thm}
Though the assumption that $f$ vanishes on a nonempty open set drastically changes the nature of the theorem but it is still sufficient for the application we have in mind. 
The main idea here is to suitably use the connection between the Carleman type condition (\ref{chernoffcond}) and approximation by polynomials proved in \cite{DJ}. In fact, we will use (\ref{chernoffcond}) to obtain completeness of Opdam hypergeometric functions in a suitable $L^2$-space (see Lemma \ref{lemgdense}). For a discussion on quasianaliticity and approximation by polynomials on Lie groups, we refer the reader to \cite{DJ2} and references therein. We also construct an example (see Example \ref{example}) to show the impossibility of proving an exact analogue of Chernoff's result Theorem \ref{ch}, (b) on symmetric spaces if we restrict ourselves only to the class of $G$-invariant differential operators on $X$. It would be interesting to explore the possibility of extending Theorem \ref{cher-symm} (and more generally, Theorem \ref{ch}, $(b)$) to more general Riemannian manifolds.
Our final result in this paper is the following analogue of the theorem of Ingham (Theorem \ref{ingrn}) on $X$. 
\begin{thm} \label{symthm}
Let $\theta:[0,\infty) \rightarrow [0,\infty)$ be a decreasing function with $\lim_{r\to\infty}\theta(r)=0$ and 
\be\label{Idefn}
I=\int_{\left\{\la\in\mathfrak a_+^*\mid\:\|\la\|_B \geq 1\right\}} \frac{\theta (\|\la\|_B)}{\|\la\|_B^d} ~d\la,
\ee
where $d=\text{rank}(X)$.
\begin{enumerate}
\item[$(a)$] Suppose $f\in L^1(X)$ and its Fourier transform $\widetilde{f}$ satisfies the estimate
\be \label{symest}
\int_{\mathfrak{a}^* \times K} |\widetilde f(\la, k)|~ e^{\theta(\|\la\|_B)\|\la\|_B}~ |{\bf c}(\la)|^{-2}d\la~dk < \infty.
\ee
If $f$ vanishes on a nonempty open set in $X$ and $I$ is infinite then $f$ is the zero function. 
\item[$(b)$] If $I$ is finite then given any positive number $L$, there exists a nontrivial $f\in C_c^{\infty}(G//K)$ supported in $\mathcal B(o, L)$ satisfying the estimate (\ref{symest}).
\end{enumerate}
\end{thm}
We refer the reader to Theorem \ref{boundedsymthm} for an exact analogue of Theorem \ref{ingrn}. Recently we could prove a result analogous to Theorem \ref{symthm} for Riemannian symmetric spaces $X$ (see \cite[Theorem 1.2]{BR}) where the function $\|\la\|_B\theta (\la)$ was assumed to be increasing (as in \cite{L2}). It turns out that with some effort, one can construct a function $\theta$ satisfying the conditions of Theorem \ref{symthm} such that $\|\la \|\theta (\la)$ is decreasing on an unbounded set of positive Lebesgue measure. This essentially proves that the results of Ingham and Levinson are independent of each other. 

This paper is organised as follows:  In the next section, we will recall the required preliminaries regarding harmonic analysis on Riemannian symmetric spaces of noncompact type. We will prove Theorem \ref{cher-symm} and Theorem \ref{symthm} in section $3$ and section $4$ respectively. 

\section{Riemannian symmetric spaces of noncompact type}
In this section, we describe the necessary preliminaries regarding semisimple Lie groups and harmonic analysis on associated Riemannian symmetric spaces. These are standard and can be found, for example, in \cite{GV, H, H1, H2}. To make the article self-contained, we shall gather only those results which will be used throughout this paper. 

Let $G$ be a connected, noncompact, real semisimple Lie group with finite centre and $\mathfrak g$ its Lie algebra. We fix a Cartan involution $\theta$ of $\mathfrak g$ and write $\mathfrak g = \mathfrak k \oplus \mathfrak p$ where $\mathfrak k$ and $\mathfrak p$ are $+1$ and $-1$ eigenspaces of $\theta$ respectively. Then $\mathfrak k$ is a maximal compact subalgebra of $\mathfrak g$ and $\mathfrak p$ is a linear subspace of $\mathfrak g$. The Cartan involution $\theta$ induces an automorphism $\Theta$ of the group $G$ and $K=\{g\in G\mid \Theta (g)=g\}$ is a maximal compact subgroup of $G$. Let $\mathfrak a$ be a maximal subalgebra in $\mathfrak p$; then $\mathfrak a$ is abelian. We assume that $\dim \mathfrak a = d$, called the real rank of $G$. Let $B$ denote the Cartan Killing form of $\mathfrak g$. It is known that $B\mid_{\mathfrak p\times\mathfrak p}$ is positive definite and hence induces an inner product and a norm $\| \cdot \|_B$ on $\mathfrak p$. The homogeneous space $X=G/K$ is a smooth manifold  with $\text{rank}(X)=d$. The tangent space of $X$ at the point $o=eK$ can be naturally identified to $\mathfrak p$ and the restriction of $B$ on $\mathfrak p$ then induces a $G$-invariant Riemannian metric $\mathsf d$ on $X$. For a given $g\in G$ and a positive number $L$ we define
\bes
{\mathcal B}(gK,L)=\{xK\mid x\in G,\:\:{\mathsf d}(gK,xK)<L\},
\ees
to be the open ball with center $gK$ and radius $L$.
We can identify $\mathfrak a$ with $\mathbb{R}^d$ endowed with the inner product induced from $\mathfrak p$ and let $\mathfrak{a}^*$ be the real dual of $\mathfrak{a}$. The set of restricted roots of the pair $(\g, \mathfrak{a})$ is denoted by $\Sigma$.  It consists of all $\alpha \in \mathfrak{a}^*$ such that
\bes
\g_\alpha = \left\{X\in \g ~|~ [Y, X] = \alpha(Y) X, \:\: \txt{ for all } Y\in \mathfrak{a} \right\},
\ees
is nonzero with $m_\alpha = \dim(\g_\alpha)$. We choose a system of positive roots $\Sigma_+$ and with respect to $\Sigma_+$, the positive Weyl chamber
$\mathfrak{a}_+ = \left\{X\in \mathfrak{a} ~|~ \alpha(X)>0,\:\:  \txt{ for all } \alpha \in \Sigma_+\right\}$. 
We denote by 
\bes
\mathfrak{n}= \oplus_{\alpha \in \Sigma_+}  ~ \mathfrak{g}_{\alpha}.
\ees 
Then $\mathfrak{n}$ is a nilpotent subalgebra of $\g$ and we obtain the Iwasawa decomposition $\g = \mathfrak{k} \oplus \mathfrak{a} \oplus \mathfrak{n}$. If $N=\exp \mathfrak{n}$ and $A= \exp \mathfrak{a}$ then $N$ is a Nilpotent Lie group and $A$ normalizes $N$. For the group $G$, we now have the Iwasawa decomposition 
$G= KAN$, that is, every $g\in G$ can be uniquely written as 
\bes
g=\kappa(g)\exp H(g)\eta(g), \:\:\:\: \kappa(g)\in K, H(g)\in \mathfrak{a}, \eta(g)\in N,
\ees 
and the map 
\bes
(k, a, n) \mapsto kan
\ees 
is a global diffeomorphism of $K\times A \times N$ onto $G$. Let $\rho=\frac{1}{2}\sum_{\alpha\in \Sigma_+}m_{\alpha}\alpha$ be the half sum of positive roots counted with multiplicity.
Let $M'$ and $M$ be the normalizer and centralizer of $\mathfrak{a}$ in $K$ respectively.
Then $M$ is a normal subgroup of $M'$ and normalizes $N$. The quotient group $W = M'/M$ is a finite group, called the Weyl group of the pair $(\g, \mathfrak{k})$. $W$ acts on $\mathfrak{a}$ by the adjoint action. It is known that $W$ acts as a group of orthogonal transformation (preserving the Cartan-Killing form) on $\mathfrak{a}$. Each $w\in W$ permutes the Weyl chambers and the action of $W$ on the Weyl chambers is simply transitive. Let $A_+= \exp{\mathfrak{a_+}}$. Since $\exp: \mathfrak{a} \to A$ is an isomorphism we can identify $A$ with $\R^d$. If $\overline{A_+}$ denotes the closure of $A_+$ in $G$, then one has the polar decomposition $G=K A K$,
that is, each $g\in G$ can be written as 
\bes
g=k_1 (\exp Y) k_2, \:\:  k_1, k_2 \in K, Y\in \mathfrak{a}.
\ees 
In the above decomposition, the $A$ component of $\mathfrak{g}$ is uniquely determined modulo $W$. In particular, it is well defined in $\overline{A_+}$. The map $(k_1, a, k_2)\mapsto k_1ak_2$ of $K\times A\times K$ into $G$ induces a diffeomorphism of $K/M\times A_+\times K$ onto an open dense subset of $G$. 
It follows that if $gK=k_1 (\exp Y)K\in X$ then
\be\label{metricexp}
\mathsf d (o, gK)=\|Y\|_B.
\ee
We extend the inner product on $\mathfrak{a}$ induced by $B$ to $\mathfrak{a}^*$ by duality, that is, we set
\bes
\langle \la, \mu \rangle =B(Y_\la, Y_\mu), \:\:\:\: \la, \mu \in \mathfrak{a}^*,  ~ Y_\la, Y_\mu \in \mathfrak{a},
\ees
where $Y_\la$ is the unique element in $\mathfrak{a}$ such that 
\bes
\la(Y) = B(Y_\la, Y), \:\:\:\: \txt{ for all } Y\in \mathfrak{a}.
\ees
This inner product induces a norm, again denoted by $\|\cdot\|_B$, on $\mathfrak{a}^*$,
\bes
\|\la\|_B = \langle \la, \la \rangle^{\frac{1}{2}}, \:\:\:\: \la \in \mathfrak{a}^*.
\ees
The elements of the Weyl group $W$ acts on $\mathfrak a^*$ by the formula
\bes
sY_{\la}=Y_{s\la},\:\:\:\:\:\:s\in W,\:\la\in\mathfrak a^*.
\ees
Let $\mathfrak{a}_\C^*$ denote the complexification of $\mathfrak{a}^*$, that is, the set of all complex-valued real linear functionals on $\mathfrak{a}$. If $\la: \mathfrak{a} \to \C$ is a real linear functional then $\Re \la: \mathfrak{a} \to \R$ and $\Im \la: \mathfrak{a} \to \R$, given by 
\beas
&&\Re \la(Y)= \txt{ Real part of } \la(Y), \:\:\:\: \txt{ for all } Y\in \mathfrak{a}, \\
&& \Im \la(Y)= \txt{ Imaginary part of } \la(Y), \:\:\:\: \txt{ for all } Y\in \mathfrak{a},
\eeas
are real-valued linear functionals on $\mathfrak{a}$ and $\la = \Re \la + i \Im \la$. The usual extension of $B$ to $\mathfrak{a}_\C^*$, using conjugate linearity is also denoted by $B$. Hence $\mathfrak{a}_\C^*$ can be naturally identified with $\C^d$ such that
\bes
\|\la\|_B= \left(\|\Re \la\|_B^2 + \|\Im \la\|_B^2\right)^{\frac{1}{2}},\:\:\:\ \la \in \mathfrak{a}_\C^*.
\ees
Through the identification of $A$ with $\R^d$, we use the Lebesgue measure on $\R^d$ as the Haar measure $da$ on $A$. As usual on the compact group $K$, we fix the normalized Haar measure $dk$ and $dn$ denotes a Haar measure on $N$. The following integral formulae describe the Haar measure of $G$ corresponding to the Iwasawa and Polar decomposition respectively.
\beas
\int_{G}{f(g)dg} &=& \int_K \int_{\mathfrak{a}}\int_N f(k\exp Y n)~e^{2\rho(Y)}\:dn\:dY\:dk,\:\:\:\:\:\: f\in C_c(G) \\ 
&=&\int_{K}{\int_{\overline{A_+}}{\int_{K}{f(k_1ak_2) ~ J(a)\:dk_1\:da\:dk_2}}},
\eeas
where $dY$ is the Lebesgue measure on $\R^d$ and 
\bes 
J(\exp Y)= c \prod_{\alpha\in \Sigma_+}\left(\sinh\alpha(Y)\right)^{m_{\alpha}}, \:\:\:\: \txt{ for } Y\in \overline{\mathfrak{a}_+},
\ees
$c$ being a normalizing constant. 
%It follows that 
%\be \label{jest} 
%J(\exp Y) \leq C e^{2\|\rho\|_B \|Y\|_B}, \:\:\:\: \txt{ for all } Y\in \overline{\mathfrak{a}_+}.
%\ee 
If $f$ is a function on $X= G/K$ then $f$ can be thought of as a function on $G$ which is right invariant under the action of $K$. It follows that on $X$ we have a $G$ invariant measure $dx$ such that 
\bes
\int_X f(x)~dx= \int_{K/M}\int_{\mathfrak{a}_+}f(k\exp Y)~J(\exp Y)~dY~dk_M,
\ees
where $dk_M$ is the $K$-invariant measure on $K/M$.  
For a sufficiently nice function $f$ on $X$, its Fourier transform $\widetilde{f}$ is defined on $\mathfrak{a}_{\C}^* \times K$ by the formula 
\be \label{hftdefn}
\widetilde{f}(\la,k) = \int_{G} f(g) e^{(i\la - \rho)H(g^{-1}k)} dg,\:\:\:\:\:\: \la \in \mathfrak{a}_{\C}^*,\:\: k \in K, 
\ee
whenever the integral exists (\cite[P. 199]{H1}). 
As $M$ normalizes $N$ the function $k\mapsto\widetilde{f}(\la, k)$ is right $M$-invariant.
It is known that if $f\in L^1(X)$ then $\widetilde{f}(\la, k)$ is a continuous function of $\la \in \mathfrak{a}^*$, for almost every $k\in K$ (in fact, holomorphic in $\la$ on a domain containing $\frak a^*$). If in addition $\widetilde{f}\in L^1(\mathfrak{a}^*\times K, |{\bf c}(\la)|^{-2}~d\la~dk)$ then the following Fourier inversion holds,
\be\label{hft}
f(gK)= |W|^{-1}\int_{\mathfrak{a}^*\times K}\widetilde{f}(\la, k)~e^{-(i\la+\rho)H(g^{-1}k)} ~ |{\bf c}(\la)|^{-2}d\la~dk,
\ee
for almost every $gK\in X$ (\cite[Chapter III, Theorem 1.8, Theorem 1.9]{H1}). Here ${\bf c}(\la)$ denotes Harish Chandra's $c$-function. Moreover, $f \mapsto \widetilde{f}$ extends to an isometry of $L^2(X)$ onto $L^2(\mathfrak{a}^*_+\times K, |{\bf c}(\la)|^{-2}~d\la~dk )$ (\cite[Chapter III, Theorem 1.5]{H1}).
%\end{enumerate}
%\end{thm}
\begin{rem} \label{cprop} 
It is known that (\cite[P. 394]{A0}, \cite[P. 117]{CGM}) for $\la\in\mathfrak a_+^*$ there exists a positive number $C$ such that 
\be
|{\bf c}(\la)|^{-2}\leq C (1+\|\la\|_B)^{~\text{dim }\mathfrak n}.\label{clambdaest}
\ee
If $\text{rank}(X)=1$, then a similar lower estimate holds (\cite{ADY}, P. 653); there exist two positive numbers $C_1$ and $C_2$ such that for all $\la\geq 1$
\be
C_1\la^{\text{dim }\mathfrak n}\leq |{\bf c}(\la)|^{-2}\leq C_2 \la^{\text{dim }\mathfrak n}.\label{clambdaestone}
\ee
\end{rem}
We now specialize to the case of $K$-biinvariant function $f$ on $G$. Using the polar decomposition of $G$ we may view a $K$-biinvariant function $f$ on $G$ as a function on $A_+$, or by using the inverse exponential map we may also view $f$ as a function on $\mathfrak{a}$ solely determined by its values on $\mathfrak{a}_+$. Henceforth, we shall denote the set of $K$-biinvariant functions in $L^1(G)$ by $L^1(G//K)$.
If $f\in L^1(G//K)$ then the Fourier transform $\widetilde{f}$ can also be written as
\be \label{hlsphreln}
\widetilde f(\la, k) =\widehat{f}(\la )= \int_Gf(g)\phi_{-\la}(g)~dg,
\ee
where 
\be \label{philambda} 
 \phi_\la(g) 
= \int_K e^{-(i\la+ \rho) \big(H(g^{-1}k)\big)}~dk,\:\:\:\:\:\:\la \in \mathfrak{a}_\C^*,  
\ee
is Harish Chandra's elementary spherical function.
We now list down some well known properties of the elementary spherical functions which are important for us (\cite{GV}, Prop 3.1.4 and Chapter 4, \S 4.6; \cite{H1}, Lemma 1.18, P. 221).
\begin{thm} \label{philambdathm}
\begin{enumerate}
\item[(1)] $\phi_\la(g)$ is $K$-biinvariant in $g\in G$ and $W$-invariant in $\la\in \mathfrak{a}_\C^*$.
\item[(2)] $\phi_\la(g)$ is $C^\infty$ in $g\in G$ and holomorphic in $\la\in \mathfrak{a}_\C^*$.
\item[(3)] For all $\la\in \overline{\mathfrak{a}_+^*}$ and $g\in G$ we have
\be
|\phi_\la(g)| \leq  \phi_0(g)\leq 1.\label{phi0}
\ee
\item[(4)] For all $Y\in \overline{\mathfrak{a}_+}$ and $\la \in \overline{\mathfrak{a}_+^*}$
\be\label{phiila}
0 < \phi_{i \la}(\exp Y) \leq e^{\la(Y)} \phi_0(\exp Y).
\ee
\item[(5)] If $\Delta$ denotes the Laplace-Beltrami operator on $X$ then 
\bes
\Delta(\phi_{\la})=-(\|\la\|_B^2+\|\rho\|_B^2)\phi_{\la},\:\:\:\:\:\: \la\in \mathfrak{a}_\C^*.
\ees
%\item[5)] Given $E, D \in \mathcal{U}(g)$ there exists a positive constant $C_{D, E}$ such that 
%\bes
%|\phi_\la(E:g:D)|\leq  C_{D, E} \left(1+\|\la\|_B \right)^{\text{deg }E + \text{deg }D}  \phi_0(g), \:\: \la \in \mathfrak{a}^*. 
%\ees
%\item[(5)] For $\la\in \mathfrak a^*$
%\bes
%\phi_{-\la}(hg) = \int_K e^{(i\la - \rho)\big(H(g^{-1}k) \big)}e^{-(i\la+\rho)\big(H(hk) \big)}~dk, \:\:\:\: g, h \in G,
%\ees
\end{enumerate}
\end{thm}
%\begin{rem}\label{hftproperitesrem}
%If $\la \in \mathfrak{a}^*$ and $f\in L^1(X)$ then
%\bes
%\int_K |\widetilde f(\la, k)|dk= \int_K \left|\int_X f(g)e^{(i\la - \rho) \left( H(g^{-1}k) \right)}dg\right|dk \leq  \int_X |f(g)|\phi_0(g)dg<\infty, 
%\ees
%by Theorem \ref{philambdathm}, 3). Hence, the function $k\mapsto \widetilde f(\la, k)$ is integrable on $K$.
%\end{rem}
For $K$-biinvariant $L^p$ functions on $G$ the following Fourier inversion formula is well known (\cite{ST}, Therem 3.3 and \cite{NPP}, Theorem 5.4): if $f \in L^p(G//K)$, $1\leq p\leq 2$ with $\hat{f}\in L^1(\mathfrak{a}^*_+ , |{\bf c}(\la)|^{-2}~d\la )$ then for almost every $g\in G$,
\be
f(g)=|W|^{-1}\int_{\frak{a}^*}\hat{f}(\la )\phi_{-\la}(g)|{\bf c}(\la)|^{-2}~d\la~dk.\label{FI}
\ee 
The spherical Fourier transform and the Euclidean Fourier transform on $\mathfrak{a}$ are related by the so-called Abel transform. For $f\in L^1(G//K)$ its Abel transform ${\mathcal A}f$ is defined by the integral 
\bes
{\mathcal A}f(\exp Y)= e^{\rho(Y)}\int_Nf((\exp Y) n)\:dn, \:\:\:\: Y\in \mathfrak{a},
\ees 
(\cite[P. 107]{GV}, \cite[p.27]{H3}). We will need the following theorem regarding Abel transform
(\cite[Prop 3.3.1, Prop 3.3.2]{GV}), which we will refer as  the slice projection  theorem.
\begin{thm} \label{Abelthms}
The map ${\mathcal A}: C_c^\infty(G//K)\to C_c^\infty(\mathfrak{a})^W$ is a bijection.
If $f\in C_c^\infty(G//K)$ then
\be \label{Abelftreln}
\mathcal{F}\big({\mathcal A}f\big)(\la)= \widehat f(\la), \:\: \la\in \mathfrak{a}^*, 
\ee
where $\mathcal{F}({\mathcal A}f)$ denotes the Euclidean Fourier transform on $\mathfrak a\cong \R^d$.
\end{thm}
\begin{rem}\label{trivial}
It is easy to see that a special case of Theorem \ref{symthm}, namely when $f\in C_c^{\infty}(G//K)$, can be proved simply by using the slice projection theorem (see \cite{BS} for a more general result in this direction). However, this approach cannot be used to prove Theorem \ref{symthm}, because if an integrable $K$-biinvariant function $f$ vanishes on an open set then it is not necessarily true that ${\mathcal A}f$ also vanishes on an open subset of $\mathfrak a$. 
\end{rem}

%\begin{rem}\label{ranalyticity}
%From a) and b) we observe that if $f\in L^1(X)$ then 
%$f*h_t= (f*h_{t/2})* h_{t/2}$
%is also real analytic as $f*h_{t/2}\in L^2(X)$. In particular, if $f\in L^1(X)$ is nonzero then $f*h_t$ (for any fixed $t\in (0, \infty)$) cannot vanish on any nonempty open subset of $X$. This follows from (\ref{hsprodrln}), the Fourier inversion ( \ref{hft}) and the fact that $\widehat{h_t}$ is nonzero everywhere on $\mathfrak a^*$.
%\end{rem}

We end this section with a short discussion on Opdam hypergeometric functions which will play a crucial role in the next section.  Let $\mathfrak{a}_{\text{reg}}$ be the subset of regular elements in $\mathfrak{a}$, that is, \begin{equation}\nonumber
\mathfrak{a}_{\text{reg}}=\cup_{\alpha\in\Sigma} (\ker \alpha)^\complement.
\end{equation}  
For $\xi\in\mathfrak{a}$, let $T_\xi$ be the Dunkl-Cherednik operator \cite[Definition 2.2]{Op}, defined for $f\in C^1(\mathfrak{a})$ and for $x\in \mathfrak{a}_{\text{reg}}$, by
\begin{equation}\nonumber
T_\xi f(x)=\frac{\partial }{\partial \xi}f(x) + \sum_{\alpha\in \Sigma_+} m_\alpha \frac{\alpha(\xi)}{1-e^{-\alpha (x)}}\left(f(x)-f(r_\alpha x)\right)-\rho (\xi) f(x),
\end{equation}
where $r_\alpha$ is the orthogonal reflection with respect to the hyperplane $\ker \alpha$ and $\frac{\partial}{\partial \xi}f$ is the directional derivative $f$ in the direction of $\xi$.
In particular $r_{\alpha}$ is an isometry of $(\frak a, \|\cdot\|_B)$.
Let $\{\xi_1, \xi_2, \cdots, \xi_d\}$ be an orthonormal basis of $\mathfrak{a}$ with respect to the inner product given by $B$. Then the Heckman-Opdam Laplacian $\mathcal L$ defined by 
\bes
\mathcal L=\sum_{i=1}^d T_{\xi_i}^2,
\ees
is the $K$-invariant part of the Laplace-Beltrami operator $\Delta$ on $X$ (\cite[Theorem 2.2, Remark 2.3]{Heckman}). 

Now, suppose that $f$ is a $K$-biinvariant smooth, function on $G$ which vanishes on the ball $\mathcal B(o,L)$ for some positive number $L$. Using the polar decomposition of $G$ we can view $f$ as a function on $A$ and hence on the Lie algebra $\frak a$. Using (\ref{metricexp}) it follows that this latter function (again denoted by $f$) vanishes on the set $\{X\in\frak a\mid \|X\|_B<L\}$ which we will continue to denote by the same symbol $\mathcal B (o,L)$. Since $r_{\alpha}$ is an isometry it follows from the expression of the Dunkl-Cherednik operator that in this case $T_{\xi}f$ also vanishes on $\mathcal B (o,L)$ for all $\xi\in\frak a$.  

The Opdam hypergeometric function $G_\lambda, \lambda\in\mathfrak{a}^\ast_\C$ is defined to be the unique analytic function on $\mathfrak{a}$ such that 
\begin{equation}\label{txi}
T_\xi G_\lambda=i\lambda(\xi) \,G_\lambda,\:\:\:\:\xi\in\mathfrak{a},\:\:\:\:\:\:\:\:G_\lambda(0)=1,
\end{equation}
((see \cite[p. 89]{Op}).
The elementary spherical function $\phi_\lambda$ and the Opdam hypergeometric function $G_\lambda$ are related by \cite[p. 89]{Op}
\begin{equation}\label{g-phi-relation}
\phi_\lambda(x)=\frac{1}{|W|}\sum_{w\in W} G_{\lambda}(wx) , \,\,\lambda\in\mathfrak{a}^\ast_\C, x\in \mathfrak{a}.
\end{equation}
However, there exists an alternative relation between $\phi_{\lambda}$ and $G_{\la}$ which is important for us \cite[(4), p. 119]{Op}: For all $x\in \mathfrak{a}$, 
\begin{equation}\label{g-phi-relation-lambda}
\phi_\lambda(x)=\frac{1}{|W|}\sum_{w\in W} g(w\lambda )G_{w\lambda}(x) , \,\,\text{ for almost every }\lambda\in\mathfrak{a}^\ast, 
\end{equation}
where \begin{equation} \label{smallg}
g(\lambda)=\prod_{\alpha\in \Sigma_+^0} \left(1-\frac{\frac{m_\alpha}{2}+ \frac{m_{\alpha/2}}{4}}{\lambda_\alpha}\right), \:\:\:\:\:\:\:\:\lambda_\alpha=\frac{\langle\lambda,\alpha\rangle}{\langle\alpha, \alpha\rangle},
\end{equation}
and $\Sigma_+^0=\{\alpha\in \Sigma_+\mid 2\alpha\not\in \Sigma\}$. We now list down a few results regarding the function $G_{\la}$ which will be needed:
\begin{enumerate} 
\item  For $\la\in\mathfrak{a^\ast}$
\be \label{estgc}
|g(\la)| |{\bf c}(\la)|^{-1}\leq C_1 + C_2 \|\la\|^p,
\ee
for some $C_1, C_2, p\geq 0$ (follows from \cite[equation (8.1)]{Op}).
\item For $\lambda\in \mathfrak{a}^\ast$, the function $G_\lambda$ is known to be bounded on $\mathfrak{a}$ (\cite[Proposition 6.1]{Op}).
\item More generally, for any polynomial $p$ of degree $N$ there exists a constant $C_p$ such that for all $\lambda\in \mathfrak{a}^\ast$, $x\in\mathfrak{a}$
\begin{equation}\label{derivative-G-lambda}
\left|p\left(\frac{\partial}{\partial x}\right) G_\lambda(x) \right| \leq C_p (1 + |\lambda|)^N \phi_0(x),
\end{equation}
where $\frac{\partial}{\partial x}G_\lambda$ is the directional derivative of $G_\lambda$ in the direction $x$   (see \cite[Proposition 3.2]{Sa}).
\end{enumerate}
\section{Chernoff's theorem for symmetric spaces}
In this section our aim is to prove the theorem of Chernoff (Theorem \ref{cher-symm}).
We start with a few results which will be needed to prove Theorem \ref{cher-symm}.
The following lemma is just a restatement of \cite[Theorem 2.3]{DJ} in view of the identification of $\mathfrak{a}^\ast$ with $\R^d$. 
\begin{lem} \label{lempolydense}
Let $\mu$ be a finite Borel measure on $\mathfrak a^*$ such that, for $m\in \N$ and $1\leq j\leq d$ the quantity $M_j(m)$, defined by
\bes
M_j(m)=\int_{\mathfrak a^*}|\la (\xi_j)|^m ~ d\mu(\la),
\ees 
is finite, where $\{\xi_1, \xi_2, \cdots, \xi_d\}$ is an orthonormal basis of $\mathfrak{a}$.
If for each $j\in \{1, \cdots, d\}$, the sequence $\{M_j(2m)\}_{m=1}^{\infty}$ satisfies the Carleman's condition  
\be \label{carlcond}
\sum_{m\in \N} M_j(2m)^{-\frac{1}{2m}}= \infty, 
\ee
then the polynomials in $\mathfrak a^*$ are dense in $L^{2}(\mathfrak{a^\ast}, d\mu)$.
\end{lem}
Given a positive number $L$ we consider the following function space
\bes
G_L(\mathfrak{a}^*)=\text{span} \big\{\chi_x: \mathfrak{a}^*\to \C \mid x\in {\mathcal B}(o, L),~ \chi_x(\la)= G_\la(x),  \:\: \la \in \mathfrak{a}^*\big\}.
\ees
\begin{lem} \label{lemgdense}
Let $\mu$ be a finite  Borel measure on $\mathfrak a^*$ such that for each $j\in \{1, \cdots, d\}$ and each $m\in \N$ the quantity $M_j(m)$ (as in Lemma \ref{lempolydense}) is finite. If for each $j\in \{1, \cdots, d\}$ the sequence $M_j(2m)$ satisfies the Carleman's condition (\ref{carlcond})
then for each $L$ positive, $G_L(\mathfrak a^*)$ is dense in $L^2(\mathfrak a^*, d\mu)$.
\end{lem}
\begin{proof}
We first note that because of finiteness of $\mu$ and boundedness of $G_\la$, for $\la\in \mathfrak a^\ast$, the space $G_L(\mathfrak a^\ast)$ is a subset of $L^2(\mathfrak a^\ast, d\mu)$. Let $f\in L^2(\mathfrak a^*, d\mu)$ be such that
\be \label{polyzero2}
\int_{\mathfrak a^*} f(\lambda) G_{\lambda}(x) ~ d\mu(\lambda)=0,  \:\:\:\: \txt{for all } x\in \mathcal B(o, L).
\ee
We define a function $F$ on $\frak a$ by 
\bes
F(x)= \int_{\mathfrak a^*} f(\lambda) G_{\lambda}(x) ~ d\mu(\lambda), \:\:\:\: \txt{for } x\in \mathfrak a.
\ees
It follows from (\ref{polyzero2}) that $F$ vanishes on the ball $\mathcal B(o, L)$. Estimate (\ref{derivative-G-lambda}) together with dominated convergence theorem implies that 
\bes
P\left(\frac{\partial}{\partial x} \right)F(x)=\int_{\mathfrak a^*}f(\lambda)~ P\left(\frac{\partial}{\partial x}\right) G_\lambda(x) ~ d\mu(\lambda),
\ees
for any polynomial $P$.
Hence, for $\alpha=(\alpha_1, \cdots,\alpha_d) \in (\N\cup \{0\})^d$ we have from the eigenvalue equation (\ref{txi}) that
\bea
T_{\xi_1}^{\alpha_1}\cdots T_{\xi_d}^{\alpha_d}F(x)&=& \int_{\mathfrak a^*}f(\lambda)~ T_{\xi_1}^{\alpha_1}\cdots T_{\xi_d}^{\alpha_d}G_\lambda(x) ~ d\mu(\lambda)\nonumber\\
&=& \int_{\mathfrak a^*}f(\lambda) ~(i\lambda(\xi_1))^{\alpha_1}\cdots(i\lambda(\xi_d))^{\alpha_d}~G_\lambda(x) ~ d\mu(\lambda).\label{eigenpoly}
\eea
Since $F$ vanishes on the ball $\mathcal B(o,L)$ so does $T_{\xi_1}^{\alpha_1}\cdots T_{\xi_d}^{\alpha_d}F$. It now follows from (\ref{eigenpoly}) by taking $x=0$, that for all $\alpha=(\alpha_1, \cdots \alpha_d)\in (\N\cup \{0\})^d$
\bes
\int_{\mathfrak a^*} (\lambda(\xi_1))^{\alpha_1}\cdots  (\lambda(\xi_d))^{\alpha_d} ~f(\lambda) ~ d\mu(\lambda)=0.
\ees
This implies that $f$ annihilates all polynomials and hence by Lemma \ref{lempolydense}, $f$ is the zero function. 
\end{proof}
We will also need the following elementary lemma.
\begin{lem}\label{andiv}
Let $\{a_n\}$ be a sequence of positive numbers such that the series $\sum_{n\in \N}a_n$ diverges. Then given any $m\in\N$,
\bes
\sum_{n\in \N} a_n^{1+\frac{m}{n}}=\infty.
\ees
\end{lem}
\begin{proof}
If $\limsup a_n$ or $\liminf a_n$ is nonzero then the result follows trivially. Hence, it suffices to prove the result for the case $\lim_{n\ra \infty}a_n=0$. Without loss of generality we can also assume that $a_n\in (0, 1)$, for all $n\in \N$. Let us define
\bes
A=\big\{n\in \N: a_n \leq \frac{1}{n^2}\big\}, \:\:\:\: B=\big\{n\in \N: a_n >\frac{1}{n^2}\big\}.
\ees
As $\sum_{n\in \N} a_n$ diverges it follows that $B$ is an infinite set.  The result now follows by observing that 
\bes
\lim_{n\ra \infty, n\in B}\frac{a_n^{1+\frac{m}{n}}}{a_n}= \lim_{n\ra \infty, n\in B} a_n^{\frac{m}{n}}= \lim_{n\ra \infty, n\in B}e^{-\frac{m}{n}\log{\frac{1}{a_n}}}=1,
\ees
as for $n\in B$
\bes
1\leq \frac{1}{a_n}\leq n^2, \txt{ and hence }\:\:  0 \leq \log{\frac{1}{a_n}} \leq 2\log n.
\ees
\end{proof}
For $f\in L^1(X)$, we define the $K$-biinvariant component $\mathcal Sf$ of $f$ by the integral
\be \label{radialization}
\mathcal Sf(x) = \int_Kf(kx)~dk, \:\: x\in X,
\ee
and for $g \in G$, we define the left translation operator $l_g$ on $L^1(X)$ by 
\bes
l_g f(x)= f(gx), \:\:  x\in X.
\ees
\begin{rem}
The operator $l_g$ is usually defined as left translation by $g^{-1}$. The reason we have defined $l_g$ differently because then it follows that ${\mathcal S}(l_gf)={\mathcal S}(l_{g_1}f)$ if $gK=g_1K$.
\end{rem}
%It is known that (\cite[Chapter III, \S 2, P. 209]{H1}) the Fourier transforms of $f$ and $l_g f$ are related by the formula
%\be\label{hfttranslatereln}
%(l_{g}f{\widetilde{)}}(\la, k) = e^{(i\la - \rho) \left(H(gk)\right)} \widetilde f(\la, \kappa(g k)).
%\ee
For a nonzero integrable function $f$, its $K$-biinvariant component $\mathcal S(f)$ may not be nonzero. However, the following lemma shows that there always exists $g\in G$ such that $\mathcal S(l_gf)$ is nonzero. 
\begin{lem}$($\cite[Lemma 4.6]{BR}$)$\label{nonzeroradiallem}
If $f\in L^1(X)$ is nonzero then for every $L$ positive, there exists $g\in G$ with $gK\in {\mathcal B}(o, L)$ such that $\mathcal S(l_gf)$ is nonzero.
\end{lem}
We now present the proof of Theorem \ref{cher-symm}.
\begin{proof}[Proof of Theorem \ref{cher-symm}] The following steps will lead to the proof of the theorem.\\
\noindent{\bf Step 1:} Using translation invariance of $\Delta$, we can assume without loss of generality that $f\in C^\infty(G/K)$ and vanishes on the ball $\mathcal B(0, L)$ for some $L$ positive. We first show that it suffices to prove the result under the additional assumption that $f$ is $K$-biinvariant. To see this, suppose $f\in C^\infty(G/K)$ vanishes on $\mathcal B(0, L)$ and satisfies the hypothesis 
(\ref{chernoffcond}).  Since $\Delta$ is left-translation invariant operator and $\Delta^m f\in L^2(G/K)$, it is easy to check that 
\bes
\Delta^m(Sf)=S(\Delta^m f), \text{ for all } m\in\N.
\ees 
Therefore
\bes
\|\Delta^m(Sf)\|_2^2= \int_G |S(\Delta^m f)(g)|^2\, dg\leq \int_G \int_K |\Delta^m f(kg)|^2\,dk dg=\|\Delta^m f\|_2^2.
\ees
Hence 
\bes
\sum_{m=0}^\infty \|\Delta^m(Sf)\|_2^{-\frac{1}{2m}} \geq \sum_{m=0}^\infty \|\Delta^m f\|_2^{-\frac{1}{2m}}=\infty.
\ees
If $f$ is not identically zero but vanishes on $\mathcal B(o, L)$, then by Lemma \ref{nonzeroradiallem}, there exists $g_0K\in \mathcal B(o, L/2)$ such that $S(l_{g_0}f)$ is non zero but  vanishes on $\mathcal B(o, L/2)$. Hence, if the theorem is true for $K$-biinvariant function, then $S(l_{g_0}f)$ must vanishes identically which is a contradiction. 

So, we assume that $f\in C^\infty(G//K)$ is such that $\Delta^m f\in L^2(G//K)$ for all $m\in \N\cup\{0\}$ and satisfies (\ref{chernoffcond}). We will show that if $f$ vanishes on $\mathcal B(o, L)$ then $f$ is the zero function.\\

\noindent{\bf Step 2:} Given such a function $f$ we define a measure $\mu$ on $\frak a^\ast$ by
\bes
\mu(E)=\int_E|\hat{f}(\la)||{\bf c}(\la )|^{-2}d\la,
\ees 
for all Borel subsets $E$ of $\frak a^\ast$. Note that the measure $\mu$ is $W$-invariant. We claim that that the space $G_L(\mathfrak{a}^\ast)$ is dense in $L^2(\mathfrak{a}^\ast, \mu)$ for any given positive number $L$. Since $\Delta^nf\in L^2(\mathfrak{a}^\ast, |{\bf c}(\la )|^{-2}d\la)$ for all $n\in\N\cup \{0\}$ and $|{\bf c}(\la )|^{-2}$ is of polynomial growth (Remark \ref{cprop}) it follows by a simple application of Cauchy-Schwarz inequality that $\mu$ is a finite measure and $G_L(\mathfrak{a}^\ast)$ is contained in $L^2(\mathfrak{a}^\ast, \mu)$. The same argument also implies that polynomials are contained in $L^2(\mathfrak{a}^\ast, \mu)$. To prove the claim it suffices, in view of Lemma \ref{lemgdense}, to show that
\bes
\sum_{m=1}^{\infty}M_j(2m)^{-\frac{1}{2m}}=\infty,
\ees
where $M_j(m)$ are as in Lemma \ref{lempolydense}. Now, for a large enough $r\in\N$ and for all $m\in\N$ we have
\beas
M_j(2m)&\leq& \int_{\mathfrak{a}^\ast}(\|\la\|^2+\|\rho\|^2)^m~|\widehat f(\la)|~|{\bf c}(\la)|^{-2}~d\la\\
&\leq & \left(\int_{\mathfrak{a}^\ast}(\|\la\|^2+\|\rho\|^2)^{(2m+2r)}~|\widehat f(\la)|^2~|{\bf c}(\la)|^{-2}~d\la \right)^{\frac{1}{2}}~ \left(\int_{\mathfrak{a}^\ast} \frac{|{\bf c}(\la)|^{-2}}{(\|\la\|^2+\|\rho\|^2)^{2r}}~d\la\right)^{\frac{1}{2}}\\
&=& A_r \|\Delta^{m+r} f\|_2,
\eeas
where
\bes
A_r=\left(\int_{\mathfrak{a}^\ast} \frac{|{\bf c}(\la)|^{-2}}{(\|\la\|^2+\|\rho\|^2)^{2r}}~d\la\right)^{\frac{1}{2}}.
\ees
Therefore,
\bes
|M_j(2m)|^{-\frac{1}{2m}}\geq  A_r^{-\frac{1}{2m}}\|\Delta^{(m+r)}f\|_2^{-\frac{1}{2m}}= A_r^{-\frac{1}{2m}} ~ \left(\|\Delta^{(m+r)}f\|_2^{-\frac{1}{2(m+r)}}\right)^{\left(1+\frac{r}{m}\right)}.
\ees
Since, $\lim_{m\ra\infty} A_r^{-\frac{1}{2m}}=1$, it follows from Lemma \ref{andiv} and the hypothesis (\ref{chernoffcond}) that
\bes
\sum_{m=1}^{\infty} M_j(2m)^{-\frac{1}{2m}}=\infty.
\ees
It follows that for each positive number $L$ the space $G_L(\mathfrak{a}^\ast)$
 is dense in $L^2(\mathfrak{a}^\ast, \mu)$. \\
 
\noindent{\bf Step 3:} In view of the relation (\ref{g-phi-relation-lambda}) between $G_{\la}$ and $\phi_{\la}$ one would expect that the previous step implies something regarding completeness of the elementary spherical functions $\phi_{\lambda}$. In this regard we consider the space 
\bes
\Phi_L(\mathfrak{a}^*)=\text{span} \big\{\la\mapsto \phi_\lambda(x) \mid x\in {\mathcal B}(o, L),~ \la \in \mathfrak{a}^*\big\},
\ees
for any given positive number $L$ and claim that $\Phi_L(\mathfrak{a}^*)$ is dense in $L^1(\mathfrak{a^\ast}, \mu)^W$. Here $L^1(\mathfrak{a^\ast}, \mu)^W$ is the $W$-invariant functions in $L^1(\mathfrak{a^\ast}, \mu)$. To prove this we consider a $W$-invariant function $h\in L^\infty(\mathfrak{a^\ast}, \mu)$ such that
\bes
\int_\mathfrak{a^\ast} h(\la)\phi_\la(x)\, d\mu(\la)=0, 
\ees
for all  $x\in {\mathcal B}(o, L)$.
By (\ref{g-phi-relation-lambda}) and the $W$-invariance of $h$ it follows that   
 \be \label{hgG}
\int_\mathfrak{a^\ast} h(\la)g(\la) G_\la(x)\, d\mu(\la)=0, 
\ee
for all  $x\in {\mathcal B}(o, L)$. We will now repeatedly apply the operators $T_{\xi}$ on the integral in (\ref{hgG}) by viewing this as a function of the variable $x\in\frak a$. To justify the differentiation under the integral it is necessary to show that for each $n\in \N\cup \{0\}$
 \bes
 \int_\mathfrak{a^\ast} |h(\la)|\,\|\la\|^n\, |g(\la)|  |G_\la(x)|\, d\mu(\la)<\infty.
 \ees
By using the estimate (\ref{estgc}) and the boundedness of $G_\lambda$, we get that
\beas
&&\int_\mathfrak{a^\ast} |h(\la)|\,\|\la\|^n\, |g(\la)|  |G_\la(x)|\, d\mu(\la)\\ 
&\leq &C\|h\|_\infty \int_\mathfrak{a^\ast} \|\la\|^n\, |g(\la)|  \, |\widehat{f}(\la)| |{\bf c}(\la)|^{-2}\,d\la\\
&\leq& C\|h\|_\infty \int_\mathfrak{a^\ast} \|\la\|^n\, \left(C_1 + C_2 \|\la\|^p\right) \, |\widehat{f}(\la)| |{\bf c}(\la)|^{-1}\,d\la\\ 
&\leq & C\|h\|_\infty \int_\mathfrak{a^\ast}( \|\la\|^2+\|\rho\|^2)^M |\widehat{f}(\la)| |{\bf c}(\la)|^{-1}\,d\la,\:\:\:\:\:\:\:\:M\in\N, M\geq n+p\\
&\leq& C\|h\|_\infty \left(\int_\mathfrak{a^\ast}(\|\la\|^2+\|\rho\|^2)^{2M+d+1} |\widehat{f}(\la)|^2 |{\bf c}(\la)|^{-2}\,d\la\right)^{\frac{1}{2}} \left(\int_\mathfrak{a^\ast}\frac{1}{(\|\la\|^2+\|\rho\|^2)^{d +1}}\,d\la\right)^{\frac{1}{2}}\\
&<&\infty,
\eeas
where $d=\text{dim}~\frak a$.
Note that in the last step we have used the assumption $\Delta^m f\in L^2(G/K)$ for all $m\in \N\cup \{0\}$. 
%Therefore by applying $T_\xi$ on both sides of (\ref{hgG}) we get 
%\bes
%\int_\mathfrak{a^\ast} h(\la)g(\la) \, i \lambda(\xi)\,G_\la(x)\, d\mu(\la)=0, 
%\ees
%for all  $x\in {\mathcal B}(o, L)$. 
For each $\alpha\in \Sigma_0^+$, we now choose $\xi_\alpha\in\mathfrak{a}$ in such a way that \bes
\lambda(\xi_\alpha)=\lambda_\alpha= \frac{\lambda(\alpha)}{\langle\alpha, \alpha\rangle},\:\:\:\:\:\: \text{ for all } \lambda\in\mathfrak{a^\ast}.
\ees
Applying the composition of the operators $T_{\xi_\alpha}$, for all $\alpha\in \Sigma_0^+$  on both sides of (\ref{hgG}) it follows that
\be \label{hgGlambdazero}
\int_\mathfrak{a^\ast} h(\la)g(\la) \, \left(\prod_{\alpha\in \Sigma_0^+} i \lambda_\alpha\right)\,G_\la(x)\, d\mu(\la)=0,
\ee
for all  $x\in {\mathcal B}(o, L)$.
From the expression of the function $g$ given in (\ref{smallg}) it is easy to see that the function  $g(\la) \left(\prod_{\alpha\in \Sigma_0^+} i \lambda_\alpha\right)$ is of polynomial growth.
Since $h$ is a bounded function it follows that  the function $h(\la)g(\la) \, \left(\prod_{\alpha\in \Sigma_0^+} i \lambda_\alpha\right)$ is in $L^2(\mathfrak{a}^\ast, \mu)$.
As $G_L(\mathfrak{a^\ast})$ is dense in $L^2(\mathfrak{a^\ast}, d\mu)$  it follows from (\ref{hgGlambdazero}) that $h=0$ for almost every $\la$.\\

\noindent{\bf Step 4:}
By the Fourier inversion (\ref{FI}) we have that for all $x\in \mathcal B(o, L)$ 
\bes
 f(x) = |W|^{-1} \int_{\mathfrak a^*} \widehat f(\la)  \phi_\la (x) |{\bf c}(\la)|^{-2} d\la=0.
\ees
This implies that for all $u \in \Phi_L(\mathfrak a^\ast)$ 
\be \label{uzero}
\int_{\mathfrak a^*} \widehat f(\lambda)  u(\la) |{\bf c}(\lambda)|^{-2} d\lambda=0.
\ee 

 As $\widehat f\in L^1(\mathfrak a^\ast, d\mu)^W$ by the completeness of $\Phi_L(\mathfrak{a}^\ast)$  in $L^1(\mathfrak a^*, d\mu)^W$ we can approximate $\overline{\widehat{f}}$ by the elements of $\Phi_L(\mathfrak{a}^\ast)$, that is, given $\epsilon > 0$, there exists $u_0\in \Phi_L(\mathfrak a^\ast)$ such that \bes
 \|\overline{\widehat f}-u_0\|_{L^1(\mathfrak a^\ast, d\mu)}< \epsilon.
 \ees 
Therefore,
\beas
\int_{\mathfrak a^\ast}|\widehat f(\la)|^2|{\bf c}(\la)|^{-2}d\la &= &  \int_{\mathfrak a^\ast}\overline{\widehat f(\la)}\widehat f(\la)|{\bf c}(\la)|^{-2}d\la\\
&=& \left| \int_{\mathfrak a^\ast}\left(\overline{\widehat f(\la)}-u_0(\la)+u_0(\la)\right)~\widehat f(\la)|{\bf c}(\la)|^{-2}d\la \right|\\
&\leq & \int_{\mathfrak a^\ast}|\overline{\widehat f(\la)}-u_0(\la)|~d\mu(\lambda) + \left|\int_{\mathfrak a^\ast} \widehat{f}(\la)u_0(\la)|{\bf c}(\la)|^{-2}d\la \right|\\
&<& \epsilon.
\eeas
It follows that $\widehat f$ is zero and hence so is $f$.
\end{proof}

\begin{rem}
\begin{enumerate}
\item We will like to point out that for rank one symmetric spaces it is possible to prove Theorem \ref{cher-symm} without appealing to Dunkl-Cherednik operator $T_\xi$ and the Opdam hypergeometric functions $G_{\lambda}$. 
To see this, note that if $\mu$ is an even, finite Borel measure on $\R$ and the sequence
\bes
M(2m)=\int_{\R}\lambda^{2m}d\mu (\lambda),\:\:\:\:\:\:\:m\in\N,
\ees
satisfies the Carleman condition (\ref{carlcond}) then by Lemma 
\ref{lempolydense} the polynomials which are even functions form a dense subspace of $L^2(\R,\mu)_e$ where
\bes
L^2(\R,\mu)_e=\{f\in L^2(\R,\mu)\mid f(\lambda)=f(-\lambda), \text{ for almost every $\lambda\in\R$}\}.
\ees
We note that given any such polynomial $P$ there exists a polynomial $Q$ such that
\bes
P(\lambda)=Q(\lambda^2+\rho^2),\:\:\:\:\:\:\text{for all $\lambda\in\R$}.
\ees
Obviously the same conclusion is not valid for $W$-invariant polynomials on $\R^d$, $d>1$ and this is the main reason why we needed to use the Dunkl-Cherednik operators and the Opdam hypergeometric functions in the proof of Theorem \ref{cher-symm}. Now, if $f\in L^2(\R,\mu)_e$ is such that for all $n\in\N\cup \{0\}$
\bes
\int_{\R}f(\lambda)(\lambda^2+\rho^2)^n d\mu (\lambda)=0,
\ees
then it follows that $f$ annihilates all polynomials which are even functions. Consequently, $f$ is the zero function. This can be used to prove that the space  $\Phi_L(\R)$ is dense in $L^2(\R,\mu)_e$ (and hence in $L^1(\R,\mu)_e$). Precisely, if $f\in L^2(\R,\mu)_e$ is such that
\bes
\int_{\R}f(\lambda)\phi_{\lambda}(x)d\mu(\lambda)=0,\:\:\:\:\:\:\text{for all $x\in B(o,L)$,}
\ees
then by defining 
\bes
h(x)=\int_{\R}f(\lambda)\phi_{\lambda}(x)d\mu(\lambda),\:\:\:\:\:\:x\in G
\ees
(as in Step 3 of the proof above) and applying the Laplacian $\Delta$ repeatedly to $h$ and putting $x=e$ we get that for all $n\in\N\cup \{0\}$
\bes
\int_{\R}f(\lambda)(\lambda^2+\rho^2)^n d\mu (\lambda)=0.
\ees
which implies that $f$ is the zero function. The rest of the proof then goes as it is.
\item It was noted in \cite{CR} that Theorem \ref{ch}, (b) fails for $d=1$ if 
$\frac{d^m}{dx^m} f(0)$ vanishes only for even natural numbers  $m$. An analogous phenomena occurs for symmetric spaces also and shows that an exact analogue of 
Theorem \ref{ch}, (b) is not true for $X$ if we restrict ourselves only to the class of $G$-invariant differential operators on $X$. In the following we will illustrate this for the $n$-dimensional real hyperbolic space $\mathbb H^n$ by constructing a nonzero square integrable function $f$ on $\mathbb H^n$ such that  
\bes \Delta^m f(x_0)=0, \:\: \txt{ for all } m\in \N\cup\{0\},
\ees 
for some $x_0\in \mathbb H^n$ and satisfies (\ref{chernoffcond}).  
\end{enumerate}
\end{rem}
We start with some preliminaries on Jacobi functions (\cite{Koornwinder}). A Jacobi function $\phi_\lambda^{(\alpha, \beta)} (\alpha, \beta, \lambda\in\C, \alpha\not=-1, -2, \cdots)$ is  the unique even $C^\infty$ function on $\R$ satisfying 
\bea
\label{eqn-1}
&&\left(\frac{d^2}{dt^2} + ((2\alpha +1)\coth t + (2\beta +1)\tanh t)\frac{d}{dt} + \lambda^2 + (\alpha +\beta +1)^2\right)\phi_\lambda^{(\alpha, \beta)}(t)=0,\\
&& \hspace{2in}\nonumber \phi_\lambda^{(\alpha, \beta)}(0)=1.
\eea
In this paper we shall assume that $\alpha\geq \beta\geq -\frac 12$. 
Let $$\Delta_{(\alpha, \beta)}=\frac{d^2}{dt^2} + \left((2\alpha + 1)\coth t + (2\beta + 1)\tanh t\right)\frac{d}{dt}.$$ Then rewriting (\ref{eqn-1}) we get that 
\bea \label{eqn-2}
&& (\Delta_{(\alpha, \beta)} +\lambda^2 + (\alpha + \beta + 1)^2)\phi_\lambda^{(\alpha, \beta)}=0,\\
&& \hspace{.5in}\nonumber \phi_\lambda^{(\alpha, \beta)}(0)=1.
\eea

The Fourier-Jacobi  transform of a suitable even function $f$ on $\R$ is defined by 
\be \label{defnjacobi}
\mathcal F^{(\alpha, \beta)} f(\lambda)=\int_0^\infty f(t)\phi_\lambda^{(\alpha, \beta)}(t) (2\sinh t)^{2\alpha + 1} (2\cosh t)^{2\beta + 1}\, dt,
\ee for all complex numbers $\lambda$, for which the right hand side is well-defined. We point out that this definition coincides with the group Fourier transform when $(\alpha, \beta)$ arises from geometric cases. We also have the  inversion and Plancherel formula for the Fourier-Jacobi transformation (see \cite[Theorem 2.2, Theorem 2.3]{Koornwinder} for the statement). 

The Real hyperbolic space $\mathbb H^n$ is defined by
\bes
\mathbb H^n=\{x\in \R^{n+1}\mid -x_1^2-x_2^2-\cdots-x_n^2 + x_{n+1}^2=1, x_{n+1}>0\}.
\ees
This is a rank one symmetric space of noncompact type 
and 
$\mathbb H^n=\mathrm{SO}(n,1)/ \mathrm{SO}(n)$.
In this particular case, 
we have (see \cite[p.212]{Bray-96}),
\bes
m_1=\dim\mathfrak{g}_{\alpha}=n-1, m_{2}=\dim\mathfrak{g}_{\alpha}=0.
\ees
It is well known (\cite[(3.4)]{Koornwinder}) that the spherical function $\phi_\lambda$ on $\mathbb H^n$ is same as the Jacobi function $\phi^{(\alpha, \beta)}_\lambda$ where 
\bea \alpha=\frac{m_1 +m_2-1}{2}=\frac{n-2}{2}, \:\:\:\:\:\: \beta=\frac{m_2-1}{2}=-\frac 12,\eea and the half sum of positive roots for $\mathbb H^n$ is given by   
\bes
\rho=\alpha + \beta +1=\frac{n-1}{2}. 
\ees
Similarly for $\mathbb H^{n+2l}, l\in \N$ the spherical function is equal to the Jacobi function $\phi^{(\alpha_l, \beta_l)}_\lambda$ where \bea \alpha_l=\frac{n+2l-2}{2},\:\:\:\:\:\: \beta_l=-\frac 12,\eea and  the half sum of positive roots for $\mathcal H^{n+2l}$ is given by 
\bes
\rho_l=\alpha_l + \beta_l +1=\frac{n+2l-1}{2}=\rho+l.
\ees

If $\xi \in \mathbb H^n$ then using the Cartan decomposition of $\mathrm{SO}(n, 1)$ we can write  $\xi= ka_t\cdot \xi_0$, where $\xi_0=(0, \cdots, 1)$, $k\in K/M=S^{n-1}$ and 
\bes
a_t= \left(\begin{array}{lll}
\cosh t & 0_{1\times n-1} & \sinh t\\
0_{1\times n-1} & I_{n-1\times n-1} & 0_{1\times n-1}\\
\sinh t & 0_{1\times n-1} & \cosh t
\end{array}\right).
\ees
We will need the following version of Hecke Bochner identity on $\mathbb H^n $ \cite[Proposition 3.3.3]{Bray}:

\begin{lem}
If $f(x)=f_0(t) Y_{l}(k)$ for $x=k a_t.\xi_0$, where $Y_l$ is spherical harmonic of degree $l$ on $K/M\cong S^{n-1}$ then
\bea \label{Hecke Bochner}
\nonumber \widetilde{f}(\lambda, k)&=&d_{n,l} Q_l(i\lambda-\rho) \left( \int_0^\infty f_0(t) \phi_{\lambda}^{H^{n+2l}}(t) (\sinh t)^{2\rho +l}\,dt\right) Y_{l}(k) \\ 
&=&  d_{n,l} Q_l(i\lambda-\rho) \mathcal F^{(\alpha_l, \beta_l)}\left(\frac{f_0}{(\sinh t)^l}\right)(\lambda) Y_{l}(k),
\eea
where $\phi_{\lambda}^{H^{n+2l}}$ is the elementary spherical function on $H^{n+2l}$, $d_{n,l}$ is some fixed constant depending only on $n$ and $l$,  $Q_l(i\lambda-\rho)$ is a polynomial in $\la$  given by
\be \label{defnq}
Q_l(i\lambda-\rho)=\prod_{m=0}^{l-1}(i\lambda-\rho-m),
\ee
                                                                                                                                                                                                                                                                                                                                                                                                                                                                                                                                                                                                                                                                                      and 
                                                                                                                                                                                                                                                                                                                                                                                                                                                                                                                                                                                                                                                                                $\mathcal F ^{(\alpha_l, \beta_l)}f$ is the Jacobi transfrom of $f$ defined in (\ref{defnjacobi}).

\end{lem}
\begin{example}\label{example}
{\em Let $h_t$ be the heat kernel on $\mathbb H^{n +2l}$ (\cite{AP}). Since $h_t$ is a $K$-biinvariant function on $\mathbb H^{n+2l}$, using polar decomposition it can be viewed as an even function on $\R$ and hence it can also be viewed as a $K$-biinvariant function on $\mathbb H^n$. We now choose a spherical harmonic $Y_{l}$ of degree $l$ such that $Y_{l}(k_0)= 0$, for some $k_0\in K/M$.  We now define a function $f$ on $\mathbb H^n$  by
\be \label{defnf}
f(\xi)=(\sinh r)^l h_1(r)Y_{l}(k),\:\:\:\:\: \txt{ for } \xi=ka_r\cdot \xi_0.
\ee It follows from the point wise estimate of the heat kernel (\cite[(3.1)]{AP}) that $f\in L^2(G/K)$ and 
\bes f(k_0a_r\cdot \xi_0)=0,
\ees for all $r$ in $[0, \infty)$. We now claim that
\bes
(\Delta^mf)(k_0a_r\cdot \xi_0)=0, \:\: \txt{ for all } r.
\ees
To prove this claim  we will show that
\be \label{mdelta}
\Delta^mf(\xi)= (\sinh r)^l \left(\Delta_{(\alpha_l, \beta_l)} + \delta\right)^mh_1(r)Y_{l}(k), 
\ee
for all  $\xi= ka_r\cdot \xi_0$,  where \bes \delta=(\rho+ l)^2-\rho^2=\rho_l^2-\rho^2. \ees
Taking  Fourier transform the left hand sides of (\ref{mdelta}) and using Hecke Bochner identity (\ref{Hecke Bochner}) we get
\bea \label{ftmdelta}
\nonumber (\Delta^mf\widetilde{)}(\la, k)&=& \left(-(\la^2+\rho^2)\right)^m\widetilde f(\la, k)\\
 &=& \left(-(\la^2+\rho^2)\right)^m d_{n,l} Q_{l}(i\la-\rho)\mathcal F^{(\alpha_l, \beta_l)}(h_1)(\la)Y_{l}(k).
\eea
 On the other hand,  using  (\ref{Hecke Bochner}) we get that the  Fourier transform of the right hand side of (\ref{mdelta}) is equal to
\beas
&&d_{n,l} Q_{l}(i\la-\rho)\mathcal F^{\alpha_l, \beta_l}\left(\left(\Delta_{(\alpha_l, \beta_l)} +\delta\right)^m h_1\right)(\la)Y_{l}(k)\\
&=& d_{n,l} Q_{l}(i\la-\rho) \left( -(\lambda^2 + \rho_l^2) +\delta\right)^m \mathcal F^{\alpha_l, \beta_l} h_1(\la)Y_{l}(k) \\
&=& d_{n,l} Q_{l}(i\la-\rho)\left(-(\la^2+\rho^2)\right)^m \mathcal F^{\alpha_l, \beta_l} h_1(\la)Y_{l}(k),
\eeas
which proves (\ref{mdelta}).
Now,  using (\ref{ftmdelta}) it follows that
\beas
\|\Delta^mf\|_{L^2(G/K)}^2 &=& \|\widetilde{(\Delta^m f)}(\la, k)\|_{L^2(\R\times K, |{\bf c}(\la)|^{-2}d\la dx)}\\
&=& d_{n,l}^2 \int_{\R}|Q_l(i\la-\rho)|^2 (\lambda^2 + \rho^2)^{2m}~|\mathcal F^{(\alpha_l, \beta_l)}h_1 (\la)|^2~|{\bf c}(\la)|^{-2}~d\la.
\eeas
Using (\ref{clambdaestone}), (\ref{defnq}) we have  
\beas 
&&  |{\bf c}(\lambda)|^{-2}\leq C |\lambda|^{n_0}, \,\,\,|\lambda|\geq 1, \text{ for some } n_0, \\
&& |Q_l(i\la-\rho)|^2\leq C(|\la|^2 + \rho^2 )^{p_0}, \text{ for some } p_0>0 \text{ and }\\
 && \mathcal F^{(\alpha_l, \beta_l)}h_1 (\la)=e^{-\lambda^2}.
\eeas
Therefore
\beas
\|\Delta^m f\|_{L^2(G/K)}^2 &\leq & C\int_{\R} \left(\la^2+\rho^2\right)^{2(m+p_0)}e^{-2\la^2}~|{\bf c}(\la)|^{-2}~d\la\\
&=& C^{2(m+p_0)} + C_1\int_{1}^\infty\left(\frac{y}{2}+\rho^2\right)^{2(m+p_0)}~\left(\frac{y}{2}\right)^{\frac{n_0-1}{2}}~e^{-y}~dy  \\
&\leq & C^{2(m+p_0)}+ C_2^{2(m+p_0)}\int_{1}^{\infty}y^{2(m+p_0)}~y^{\frac{n_0-1}{2}}~e^{-y}~dy\\
&=&C^{2(m+p_0)}+ C_2^{2(m+p_0)} \Gamma\left(2m+2p_0+\frac{n_0+1}{2}\right)\\
&\leq & C_0^{2(m+p_0)}~\Gamma\left(2m+2p_0+\frac{n_0+1}{2}\right).
\eeas
Consequently,
\beas
\sum_{m\in \N}\|\Delta^mf\|_2^{-\frac{1}{2m}}\geq \sum_{m\in\N}C_0^{\frac{-2(m+p_0)}{4m}}\Gamma\left(2m+2p_0+\frac{n_0+1}{2}\right)^{-\frac{1}{4m}}
\eeas
Now,  using the fact that (\cite[p. 30]{PK})
\bes
\lim_{n\ra \infty}\frac{\Gamma(n+\alpha)}{\Gamma (n)n^\alpha}=1, \:\:\:\: \txt{ for } \alpha\in \C,
\ees
it follows that
\beas
\|\Delta^mf\|_2^{-\frac{1}{2m}}&\geq& C^{\frac{-2(m+p_0)}{4m}}\left(\Gamma(2m)\right)^{-\frac{1}{2m}}~(2m)^{-\frac{4p_0+n_0+1}{8m}}\\
&\geq&  C^{\frac{-2(m+p_0)}{4m}}(2m)^{-\frac{1}{2}}~(2m)^{-\frac{4p_0+n_0+1}{8m}}.
\eeas
Hence, for large $m$,
\bes
\|\Delta^m f\|_2^{-\frac{1}{2m}}\geq C (2m)^{-\frac{4m + 4p_0+n_0+1}{8m}},
\ees
and 
\bes
(2m)^{\frac{4m+ 4p_0+n_0+1}{8m}}\leq (2m)^{\frac{8m}{8m}}=2m.
\ees
Therefore,
\bes
\sum_{m\geq m_0}\|\Delta^m f\|_2^{-\frac{1}{2m}}\geq \sum_{m\geq m_0}\frac{1}{2m}=\infty.
\ees
This shows that the function $f$ satisfies (\ref{chernoffcond}) and $\Delta^m f(k_0 a_r.\xi_0)$ is zero for all $m\in \N\cup \{0\}$ and $r$ in $[0, \infty]$.}
\end{example}

\section{Ingham's theorem for symmetric spaces}
In this section we will prove Ingham's theorem (Theorem \ref{symthm}), using Theorem \ref{cher-symm}.
\begin{proof}
[Proof of Theorem \ref{symthm}.]
We will first prove  part $(b)$ by reducing matters to $\R^d$ with the help of Abel transform $\mathcal A$.  Since the integral in (\ref{Idefn}) is finite, we have  
\bes
\int_{1}^{\infty} \frac{\theta(r)}{r} ~ dr < \infty.
\ees
Since $\theta$ is decreasing by part $(b)$ of Theorem \ref{ingrn} for $L$ positive, there exists a nontrivial radial function $h_0 \in C_c^\infty(\R^d)$ with $\txt{supp}~h_0 \subseteq B(0,L/2)$ such that 
\be \label{h0hatest}
|\mathcal F{h_0}(\xi)|\leq Ce^{-\|\xi\|\theta(\|\xi\|)}, \:\:\:\: \txt{ for all } \xi \in \R^d.
\ee
Since $h_0$ is a radial function on $\R^d$, it can be thought of as a $W$-invariant function on $A \cong \R^d$. Hence by Theorem \ref{Abelthms}, there exists $h \in C_c^{\infty}(G//K)$ such that ${\mathcal A}(h) = h_0$ with $\txt{supp } h\subseteq \mathcal B(o, L/2)$. For a nontrivial $\phi \in C_c^\infty(G//K)$ with support contained in $\mathcal B(o, L/2)$, we consider the function $f= h*\phi \in C_c^\infty(G//K)$. It follows from Paley-Wiener theorem (\cite[Theorem 7.1, Chapter IV]{H2}) that the support of $f$ is contained in $\mathcal B(o, L)$. Using the slice projection theorem (Theorem \ref{Abelthms}) and the estimate (\ref{h0hatest}) it follows that 
\bea \label{phihatestb}
&& \int_{\mathfrak{a}^*}|\widehat f(\la)|~ e^{\|\la\|_B\theta(\|\la\|_B)}~ |{\bf c}(\la)|^{-2}~ d\la \nonumber\\
&= & \int_{\mathfrak{a}^*}|\widehat h(\la)|~|\widehat \phi(\la)| ~ | e^{\|\la\|_B\theta(\|\la\|_B)}~ |{\bf c}(\la)|^{-2}~ d\la \nonumber \\
&= & \int_{\mathfrak{a}^*}|\mathcal F  h_0(\la)|~|\widehat \phi(\la)| ~ | e^{\|\la\|_B\theta(\|\la\|_B)}~ |{\bf c}(\la)|^{-2}~ d\la \nonumber \\
&\leq & C \int_{\mathfrak{a}^*}|\widehat \phi(\la)| ~ |{\bf c}(\la)|^{-2}~ d\la.
\eea
Since, $\widehat \phi$ is a Schwartz function on $\mathfrak{a}^\ast$ it follows from the estimate (\ref{clambdaest}) of $|{\bf c}(\lambda)|^{-2}$ that the integral in (\ref{phihatestb}) is finite and consequently, $\widehat f$ satisfies the condition (\ref{symest}). This completes the proof of part (b).

We will prove part (a) under the additional assumptions that $f$ is continuous, $K$-biinvariant and vanishes on an open ball centered at $o$. The general case then can be deduced from this case by mimicking the arguments given in the proof of \cite[Theorem 1.2, steps 1-2]{BR}.  

So, we assume that $f$ is $K$-biinvariant, continuous, integrable function which vanishes on $\mathcal B(o, L)$ and satisfies the hypothesis (\ref{symest})
\be \label{bikest}
\int_{\mathfrak a^\ast}|\widehat f(\la)|~e^{\|\la\|\theta(\|\la\|_B)}~|{\bf c}(\la)|^{-2}~d\la< \infty.
\ee
We observe from (\ref{bikest}) that $\widehat f\in L^1(\mathfrak a^\ast, |{\bf c}(\la)|^{-2} d\la)$. As $f$ is an integrable function, $\widehat f$ is a bounded function and hence from (\ref{bikest}) it follows that
\be\label{bikest1}
\int_{\mathfrak{a}^* } |\widehat f(\la)|^2~ e^{\theta(\|\la\|_B)\|\la\|_B}~ |{\bf c}(\la)|^{-2}d\la 
%\leq \|\widehat{f}\|_\infty \int_{\mathfrak{a}^*} |\widehat f(\la)|~ e^{\theta(\|\la\|_B)\|\la\|_B}~ |{\bf c}(\la)|^{-2}d\la 
< \infty.
\ee

We now consider the following two cases as in \cite{I}.

\noindent
{\em Case I:} Suppose $\theta$ satisfies the inequality
\be \label{thetaest}
\theta(r)\geq \frac{4}{\sqrt r},\:\:\:\: \txt{for } r\geq 1.
\ee

From (\ref{bikest}) and (\ref{thetaest}), it follows that 
\be \label{bikestl2}
\int_{\mathfrak a^*}|\widehat f(\la)|~ e^{4 \sqrt{\|\la\|_B}}~|{\bf c}(\la)|^{-2}d\la < \infty,
\ee
and hence  in particular \bes 
\int_{\mathfrak a^*}|\widehat f(\la)|~ \|\la\|_B^N~|{\bf c}(\la)|^{-2}d\la < \infty,
\ees
for all $N\in \N$. It follows that $f\in C^\infty(G//K)$.
To apply Theorem \ref{cher-symm} we need to verify condition (\ref{chernoffcond}).
In this regard, an application of (\ref{bikest1}) implies that
\bea
\|\Delta^mf\|_{L^2(G//K)} &= &  \left(\int_{\mathfrak a^*} \left(\|\la\|_B^2+\|\rho\|_B^2\right)^{2m} ~ |\widehat f(\la)|^2~|{\bf c}(\la)|^{-2}~ d\la\right)^{\frac12} \nonumber\\ 
&\leq & \sup_{\la\in \mathfrak a^*}\left(\|\la\|_B^2+\|\rho\|_B^2\right)^{m} e^{-\frac{\|\la\|\theta(\|\la\|_B}{2}} \left(\int_{\mathfrak a^*} |\widehat f(\la)|^2~ e^{\|\la\|\theta(\|\la\|_B)}~ |{\bf c}(\la)|^{-2}~ d\la \right)^{\frac{1}{2}}\nonumber\\
&\leq & C \sup_{r\in (0, \infty)} \left(\|\rho\|_B^2+r^2\right)^m e^{-\frac{r\theta(r)}{2}}, \label{mjkest}
\eea
and the latter quantity is finite by (\ref{thetaest}).
In particular $\Delta^m f\in L^2(G//K)$ for all $m\in \N\cup \{0\}$. From now on we shall consider $m\geq \max\{2, \|\rho\|_B\}$. To estimate the $L^\infty$ norm (\ref{mjkest}) let us define
\bes
g_m(r) = \left(\|\rho\|_B^2+r^2\right)^m ~ e^{-\frac{r\theta(r)}{2}}, \:\:\:\: r\in [0, \infty).
\ees 
Then 
\bes
\|g_m\|_{L^\infty[0, \infty)} \leq \|g_m\|_{L^\infty[0,1]} + \|g_m\|_{L^\infty[1, m^4]} + \|g_m\|_{L^\infty(m^4, \infty)}.
\ees
If $r\in (m^4, \infty)$ then by (\ref{thetaest}) we have 
\bes
g_m(r) \leq (\|\rho\|_B^2+r^2)^m~ e^{-2\sqrt r} \leq 2^m r^{2m}e^{-2\sqrt r} := \gamma_m(r) ~ (\txt{say}).
\ees
The function  $\gamma_m$ attains its maximum at $r= 4m^2$. As $m^4 \geq 4m^2$ and $\gamma_m$ is decreasing on $(4m^2, \infty)$ we have 
\be
\|g_m\|_{L^\infty(m^4, \infty)} \leq \|\gamma_m\|_{L^\infty(m^4, \infty)} = 2^m  m^{8m} e^{-2m^2}=(2 m^8 e^{-2m})^m \label{gkest1}.
\ee
Also 
\be
\|g_m\|_{L^\infty[0, 1)}\leq \left(1 + \|\rho\|_B^2 \right)^m.
\ee
For  $r\in [1, m^4]$, as $\theta$ is a decreasing function,
\bes
g_m(r) \leq \left(\|\rho\|_B^2 + r^2\right)^m e^{-\frac{\theta(m^4) r}{2}}\leq \left(1 + \|\rho\|_B^2\right)^m r^{2m} e^{-\frac{\theta(m^4) r}{2}}:= \eta_m(r).
\ees
The function $\eta_m$ attains its maximum at $r= 4m/\theta(m^4)$. As $m\geq 2$ we have
\be \label{gkest2}
\|\eta_m\|_{L^\infty[1, m^4]} \leq \left(1 + \|\rho\|_B^2 \right)^m\bigg(\frac{4m}{\theta(m^4)}\bigg)^{2m}e^{-2m}
\leq \left(1 + \|\rho\|_B^2\right)^m \bigg(\frac{4m}{\theta(m^4)}\bigg)^{2m}. 
\ee
Since the right-hand side of (\ref{gkest1}) goes to zero as $m$ goes to infinity and for all large $m\in \N$
\bes
\left(\frac{4m}{\theta(m^4)}\right)^{2m} \geq \left(\frac{4m}{\theta(1)}\right)^{2m} > 1,
\ees
we have for all large $m \in \N$
\bes
\|g_m\|_{L^\infty[0, \infty]} \leq 3 \left(1 + \|\rho\|_B^2 \right)^m \left(\frac{4m}{\theta(m^4)}\right)^{2m}.
\ees
Therefore, using  inequality   (\ref{mjkest}) it follows from above that for all large $m \in \N$ and a positive number $C$
\be\label{mkest2}
\|\Delta^mf\|_{L^2(G//K)}\leq \bigg(\frac{C m}{\theta(m^4)}\bigg)^{2m}.
\ee
Applying the change of variable $\|\la\|_B=p^4$ in the integral (\ref{Idefn}) defining $I$, it follows that 
\bes
\int_{1}^{\infty}\frac{\theta (p^4)}{p} ~dp=\infty.
\ees
As $\theta$ is decreasing in $[0, \infty)$ this, in turn, implies that
\bes
\sum_{m\in \N}\frac{\theta(m^4)}{m}= \infty.
\ees 
The inequality (\ref{mkest2}) then implies that 
\bes
\sum_{m\in \N}\|\Delta^mf\|_{L^2(G//K)}^{-{\frac{1}{2m}}}= \infty.
\ees 
Since $f$ vanishes on $\mathcal B(o, L)$, it follows from Theorem \ref{cher-symm} that $f$ vanishes identically. This completes the proof under the assumption (\ref{thetaest}) on the function $\theta$.

\noindent
{\em Case II.} We now consider the general case, that is, $\theta$ is any nonnegative function decreasing to zero at infinity. Again, as in \cite{I}, we define 
\bes
\theta_1(r)= \frac{8}{\sqrt{|r| +1}}, \:\:\:\: r \in [0, \infty).
\ees 
It is clear that the integral $I$ in (\ref{Idefn}) is finite if $\theta$ is replaced by $\theta_1$. Hence, by case (b) there
exists a nontrivial $f_1 \in C_c^\infty(G//K)$ such that $\txt{supp} f_1 \subseteq \mathcal B(o, L/2)$ and satisfies the  estimate
\be\label{f_1-est}
|\widehat f_1(\la)| \leq Ce^{-\|\la\|_B \theta_1(\|\la\|_B)},\:\:\:\: \la \in \mathfrak a^*.
\ee
We now consider the function $h = f * f_1 \in L^1(G//K)$. Since $f$ vanishes on $\mathcal B(o, L)$, the function $h$ vanishes on the open set $\mathcal B(o, L/2)$. Indeed, if $g_1K\in \mathcal B(o, L/2)$ then for all $gK\in \mathcal B(o, L/2)$ it follows by using $G$-invariance of the Riemannian metric $\mathsf d$ that
\bes
\mathsf d(o, g_1gK)\leq \mathsf d(o, g_1K)+ \mathsf d(g_1K, g_1gK)< L.
\ees 
that is, $g_1gK\in \mathcal B(o, L)$. This implies that $f(g_1g)$ is zero for all $gK\in \mathcal B(o, L/2)= \txt{supp } f_1$ and hence
\bes
f\ast f_1(g_1)= \int_G f(g_1g) f_1(g^{-1})~dg =\int_{\txt{supp } f_1} f(g_1g) f_1(g^{-1})~dg=0.
\ees
We observe that
\bes
\theta(r) + \theta_1(r) \geq \frac{8}{\sqrt{r +1}}, \:\:\:\: \txt{for } r \geq 1.
\ees
Using the estimate (\ref{f_1-est}) and the hypothesis (\ref{bikest}) it follows that
\bes
\int_{\mathfrak{a}^\ast}|\widehat h(\la)| e^{\|\la\|_B\big(\theta(\|\la\|_B)+\theta_1(\|\la\|_B)\big)}|{\bf c}(\lambda)|^{-2}\,d\lambda <\infty.
\ees
Therefore $h$ satisfies all the conditions used in case $I$ and hence is the zero function. This implies that 
\bes
\widehat h(\la) = \widehat f(\la) \widehat{f_1}(\la) = 0,  
\ees
for almost every $\la \in \mathfrak a^*$. Since $\widehat f_1$ is a real analytic function, it follows that $f$ vanishes identically. 
\end{proof}

\begin{rem} 
\begin{enumerate}
\item It is worth pointing out that Ingham's theorem (Theorem \ref{symthm}) can also be proved directly using Lemma \ref{lemgdense}, without appealing to Theorem \ref{cher-symm}. For the sake of completeness we sketch the line of argument for part $(a)$ of Theorem \ref{symthm}. We assume that $f$ is $K$-biinvariant, vanishes on $\mathcal B(o, L)$, satisfies the hypothesis (\ref{symest})     and the function $\theta$ satisfies the estimate (\ref{thetaest}) with $I=\infty$.
 In this case one can easily show that that the measure $\mu$ defined in Step 2 of the proof of Theorem \ref{cher-symm} is again a finite $W$-invariant measure on $\mathfrak{a^\ast}$.
Using (\ref{bikest}), (\ref{thetaest}) and the arguments used in verifying  (\ref{chernoffcond}), one can show that 
\be \label{mjdefn}
\sum_{k\in \N} M_j (2k)^{-\frac{1}{2k}}=\infty,
\ee
where $M_j(k)$ is as given in Lemma \ref{lempolydense}.
By the Fourier inversion (\ref{hft})  we get that
\beas 
 f(x) = |W|^{-1} \int_{\mathfrak a^*} \widehat f(\lambda) ~ \phi_\lambda(x)~ |{\bf c}(\lambda)|^{-2} ~d\lambda
= 0, \:\:\:\: \txt{if } x\in \mathcal B(o, L).
\eeas

This implies that for all $u \in \Phi_L(\mathfrak a^\ast)$ 
\be \label{uzero}
\int_{\mathfrak a^*} \widehat f(\lambda) ~ u(\la)~ |{\bf c}(\lambda)|^{-2} ~d\lambda=0.
\ee
As in Step 3 of the proof of Theorem \ref{cher-symm}, it can be shown that $\Phi_L(\mathfrak{a}^\ast)$ is dense in $L^1(\mathfrak a^*, d\mu)^W$. Therefore we can approximate $\overline{\widehat{f}}$ by the elements of $\Phi_L(\mathfrak{a}^\ast)$ and hence by (\ref{uzero}) we get that $f$ is identically zero.

\item It is easy to see that part $(a)$ of Theorem \ref{symthm} remains true if the integral $I$ in (\ref{ingcond}) is replaced by the integral
\bes
\int_{\{\la\in \mathfrak a_+^*\mid\: \|\la\|_B\geq 1\}}\frac{\theta(\|\la\|_B)}{\|\la\|_B^{\eta}}|{\bf c}(\la )|^{-2}d\la,
\ees
where $\eta=d+\text{dim }\mathfrak n$, is the dimension of the symmetric space $X$. This follows from the estimate (\ref{clambdaest}) 
of $|{\bf c}(\la)|^{-2}$ as
\beas
\int_{1}^{\infty} \frac{\theta(r)}{r}  dr &=& C\int_{\{\la\in \mathfrak a_+^*\mid\:\|\la\|_B\geq 1\}}\frac{\theta(\|\la\|_B)}{\|\la\|_B^{d}}d\la\\
&=& C\int_{\{\la\in \mathfrak a_+^*\mid\: \|\la\|_B\geq 1\}}\frac{\theta(\|\la\|_B)}{\|\la\|_B^{\eta}}\|\la\|_B^{\text{dim }\mathfrak n}d\la\\
&\geq & C\int_{\{\la\in \mathfrak a_+^*\mid\: \|\la\|_B\geq 1\}}\frac{\theta(\|\la\|_B)}{\|\la\|_B^{\eta}}|{\bf c}(\la )|^{-2}d\la=\infty.
\eeas
Moreover, because of the estimate (\ref{clambdaestone}), part $(b)$ of Theorem \ref{symthm} also remains true in this case if the real rank of $G$ is one.
\end{enumerate}

\end{rem}
 
An $L^p$ version of Theorem \ref{symthm} can actually be proved by using Theorem \ref{symthm} itself. To illustrate this we prove an $L^\infty$ version of the above theorem which can be thought of as an exact analogue of Theorem \ref{ingrn}.
\begin{thm}\label{boundedsymthm}
Let $\theta$ and $I$ be as in Theorem \ref{symthm}.
\begin{enumerate}
\item[$(a)$] Suppose $f\in L^1(X)$ and the Fourier transform $\widetilde f$ satisfies the estimate
\be \label{Linftyest}
|\widetilde{f}(\la, k)| \leq C e^{-\|\la\|_B\theta(\|\la\|_B)}, \:\:\:\: \txt{ for all } \la\in \mathfrak{a}^*, k\in K.
\ee
If $f$ vanishes on a nonempty open subset of $X$ and $I$ is infinite, then $f=0$. 
\item[$(b)$] If $I$ is finite then given any $L>0$, there exists a nontrivial $f\in C_c^{\infty}(G//K)$ supported in $\mathcal B(o, L)$ satisfying the estimate (\ref{Linftyest}).
\end{enumerate}
\end{thm}
\begin{proof}
As in Theorem \ref{symthm}, it suffices to prove the theorem for $f\in L^1(G //K)$ vanishing on an open ball of the form ${\mathcal B}(o,L)$ such that $\widehat f$ satisfies the estimate
\bes
|\widehat{f}(\la)| \leq C e^{-\|\la\|_B\theta(\|\la\|_B)}, \:\:\:\: \txt{ for all } \la\in \mathfrak{a}_+^*.
\ees
We choose a nonzero $\phi \in C_c^{\infty}(G//K)$ with $\txt{supp }\phi \subseteq {\mathcal B}(o, L/2)$ and consider the function $f* \phi $. Since $f$ vanishes on ${\mathcal B}(o, L)$ and the support of the function $\phi$ is contained in ${\mathcal B}(o, L/2)$ it follows as before that $f\ast\phi$ vanishes on ${\mathcal B}(o, L/2)$. Now,
\beas
&&\int_{\mathfrak{a}^*} |\widehat{f\ast\phi} (\la)|~ e^{\|\la\|_B\theta(\|\la\|_B)}~ |{\bf c}(\la)|^{-2}~ d\la \\
&=& \int_{\mathfrak{a}^*} |\widehat\phi(\la)|~ |\widehat f(\la)|~ e^{\|\la\|_B\theta(\|\la\|_B)}~ |{\bf c}(\la)|^{-2}~ d\la \\
&\leq & C \int_{\mathfrak{a}^*} |\widehat\phi(\la)| ~ |{\bf c}(\la)|^{-2}~ d\la < \infty.
\eeas
It now follows from Theorem \ref{symthm} that $f*\phi$ is zero almost everywhere. Since $\widehat{\phi}$ is nonzero almost everywhere we conclude that $\widehat f$ vanishes almost everywhere on $\mathfrak{a}^*$ and so does $f$. To prove part $(b)$ we observe that if $I$ is finite then the function $h$ constructed in the proof of Theorem \ref{symthm}, $(b)$ satisfies the estimate (\ref{Linftyest}). 
\end{proof}
\begin{rem}\label{finalrem}
Theorem \ref{cher-symm} and Theorem \ref{symthm} can be proved in the context of Damek-Ricci spaces \cite{ADY} by similar arguments.  It will be interesting to see whether both the theorems  hold true for hypergeometric transforms associated to root systems \cite{NPP, Sa}. 
\end{rem}

\end{document}